\documentclass{article}

\usepackage{arxiv}
\usepackage{amsmath,amsthm,amssymb}
\usepackage{epsfig,minibox,multirow,makecell,tikz-cd,comment,xcolor,graphics}
\usepackage{changepage}
\usepackage{hyperref}      
\usepackage{thmtools}
\usepackage{amsfonts}       
\usepackage[verbose=true]{microtype}
\usepackage{graphicx}
\usepackage{tikz}
\usepackage{mathtools}
\usetikzlibrary{matrix,positioning}
\usepackage{subcaption}
\usepackage{thm-restate}

\usepackage{algorithm}
\usepackage{algorithmic}

\usepackage{hyperref}
\usepackage{cleveref}

\usepackage[shortlabels]{enumitem}
\usepackage{rotating}
\usepackage{fix-cm}
\usepackage{relsize}
\tikzset{fontscale/.style = {font=\relsize{#1}}
    }

\usetikzlibrary{decorations.markings}
\usetikzlibrary{calc}

 \usepackage{tikz}
 \usetikzlibrary{arrows}
 \usepackage{listings}

\theoremstyle{plain}

\newtheorem{theorem}{Theorem}[section]

\newtheorem{proposition}[theorem]{Proposition}

\theoremstyle{definition}

\theoremstyle{remark}

\usepackage[verbose=true,letterpaper]{geometry}
\AtBeginDocument{
  \newgeometry{
    textheight=8.25in,
    textwidth=6.151in,
    top=1.3in,
    headheight=14pt,
    headsep=25pt,
    footskip=30pt
  }
}
\usepackage{fancyhdr}
\fancyheadoffset{0pt} 
\title{Schwartz duality for singularly perturbed nonlinear differential equations with Chebyshev spectral method\thanks{Preprint. Under review.\\}}

\author{
 \textsuperscript{ab}Eunwoo Heo
 \and 
 \textsuperscript{bc}Kwanghyuk Park 
 \and
 \textsuperscript{abc}Jae-Hun Jung\thanks{Corresponding author \\ \\ Email: hew0920@postech.ac.kr(Eunwoo Heo), pkh0219@postech.ac.kr(Kwanghyuk Park), 
 jung153@postech.ac.kr(Jae-Hun Jung)}}

\affiliations{\textsuperscript{a} Department of Mathematics, Pohang University of Science and Technology, Pohang 37673, Korea; \\
\textsuperscript{b} POSTECH Mathematical Institute for Data Science (MINDS), Pohang University of Science and Technology, Pohang 37673, Korea; \\
\textsuperscript{c} Graduate School of Artificial Intelligence, Pohang University of Science and Technology, Pohang 37673, Korea; \\
}
\begin{document}

\setlength{\abovedisplayskip}{6pt} 
\setlength{\belowdisplayskip}{6pt} 

\maketitle

\begin{abstract}
Singularly perturbed differential equations with a Dirac delta function yield discontinuous solutions. Therefore, careful consideration is required when using numerical methods to solve these equations because of the Gibbs phenomenon. 
A remedy based on the Schwartz duality has been proposed, yielding superior results without oscillations. 
However, this approach has been limited to linear problems and still suffers from the Gibbs phenomenon for nonlinear problems.
In this note, we propose a consistent yet simple approach based on Schwartz duality that can handle nonlinear problems.
Our proposed approach utilizes a modified direct projection method with a discrete derivative of the Heaviside function, which directly approximates the Dirac delta function. 
This proposed method effectively eliminates Gibbs oscillations without the need for traditional regularization and demonstrates uniform error reduction.
\end{abstract}

\section{Introduction}
\label{intro}
Singularly perturbed partial differential equations (PDEs) frequently appear in various applications, especially for modeling phenomena where small changes in certain parameters can lead to significant variations in the solution, such as in boundary layers or reaction-diffusion processes. 
However, solving these PDEs numerically with guarantees of stability and non-oscillatory solution behavior is challenging due to the presence of singular or steep solution gradients. Various numerical methods have been developed to address these equations, including  \cite{negero2023novel,ahmed2024wavelets,elango2021finite, WANG20123403, SEVER20119966}, to name a few.
In particular, singular differential equations that contain Dirac delta function-type source terms are challenging to solve numerically. 
The most common approach relies on regularization methods to approximate the Dirac delta function, which includes methods such as the immersed boundary method \cite{YANG20097821, GUO2015529}, the volume of fluid method \cite{Lee2012} and the level set method \cite{SMEREKA200677}. 
However, when high-order spectral methods are used, the challenges increase, as spectral decomposition with orthogonal polynomials such as Chebyshev and Legendre polynomials is highly sensitive to solution regularity. 
Consequently, applying spectral approximation to differential equations with the Dirac delta function can lead to spurious oscillations, degrading accuracy and sometimes stability for nonlinear problems. 
Although these spectral methods are susceptible to oscillation, they offer the advantage of superior convergence accuracy compared to the finite difference method if used successfully.
Regularization of the Dirac delta function could prevent instability but often degrades accuracy. 
In \cite{jung2009note,YANG2017205}, Schwartz duality was used to achieve exact cancellation on the grid points and obtain numerically stable solutions with spectral accuracy. However, this approach was limited to linear problems, and oscillations persist for nonlinear problems.

In this note, we propose a consistent yet simple method to address these issues in nonlinear problems, for which we consider the following PDE for $u(x,t): [-1,1]\times [0,\infty) \rightarrow \mathbb{R}$
\begin{equation}
u_t + f_x(u)
= \delta(x-c),
\end{equation}
where $f(u)$ is, in general, a nonlinear function of $u$ and the singular source term is located at $x = c$ with $c \in (-1, 1)$.
To demonstrate the effectiveness of the proposed method, we focus on the Burgers' equation with 
$f(u) = \frac{1}{2}u^2$ and the initial and boundary conditions given by $u(x,0)= \beta > 0$ and $u(-1,t) = \beta$, respectively. 
Without loss of generality, we only consider the positive constant initial condition for simplification. For $\beta>0$, the left boundary condition should be assumed because of the well-posedness of the Burgers' equation. Similarly, if $\beta < 0$, the problem enforces the right boundary condition for the well-posedness.

In our proposed approach, we follow the same principles of Schwartz duality as in \cite{jung2009note} and employ the differentiation of the Heaviside function to represent the Dirac delta function using Chebyshev polynomials. While the effectiveness and exactness were demonstrated in \cite{jung2009note}, persistent oscillations were observed in equations with nonlinear convection terms. 
In this note, we observe that the oscillations are caused by the coupling between the nonlinear term and the solution. 
The key idea of the proposed method is to transform the Dirac delta function by scaling it with the nonlinear term, thus representing the Dirac delta function in a consistent manner with the nonlinear term coupled to the solution.

In Section~\ref{sec:Schwartz duality with Chebyshev polynomials}, we explain the Schwartz duality approximation and examine the Chebyshev polynomial for modeling singular sources. 
In Section~\ref{sec:Nonlinear differential equations perturbed with Dirac delta function}, we examine time-independent and time-dependent differential equations with singular source terms.
In Section~\ref{sec:A consistent approach: the proposed method}, we introduce a new method to mitigate oscillations in nonlinear equations.
In Section~\ref{sec:Numerical results}, we explore a variety of problem settings through numerical investigations. 
This includes solving higher-degree nonlinear equations, addressing equations with multiple source terms, focusing on two-dimensional equations, and discussing the generalization to finite difference methods with uniform points. 
In Section~\ref{sec:Concluding remark}, we summarize the findings and future research directions.

\section{Schwartz duality with Chebyshev polynomials}
\label{sec:Schwartz duality with Chebyshev polynomials}

For any real number \(c \in \mathbb{R}\), the Dirac delta function \(\delta_c = \delta(x-c) : \mathcal{S} \rightarrow \mathbb{R}\) is defined as 
for any test function \(\phi \in \mathcal{S}\), 
\(
\delta_c(\phi) = \int_{\mathbb{R}} \phi \; \delta(x-c) dx = \phi(c),
\)
where \(\mathcal{S}\) denotes the Schwartz space on \(\mathbb{R}\).
The Dirac delta function is commonly used in modeling phenomena involving singular characteristics for various applications such as \cite{severino2011, yu2023,lee2021regularization}. 
However, its numerical approximation is not well defined. When high-order methods are employed, the difficulty increases. 
We utilize Schwartz duality to approximate the Dirac delta function.
Elements of the Schwartz space $\mathcal{S}$ vanish at $\pm\infty$.
We refer to the elements of the dual space of Schwartz space, $\mathcal{S}^*$, as distributions. 
For a distribution $d \in \mathcal{S}^*$, its derivative $d'\in \mathcal{S}^*$ is defined such that $d'[\phi] = -d[\phi']$ for all $\phi \in \mathcal{S}$. 
This definition is motivated by the integration of test functions in $\mathcal{S}$. Any function $f \in \mathcal{S}$ can also be interpreted as an element of $\mathcal{S}^*$ through the expression $f[\phi] = \int_{-\infty}^{\infty} f\phi \, dx$ for $\phi \in \mathcal{S}$. 
By using integration by parts, we find $f'[\phi] = \int_{-\infty}^{\infty} f'\phi \, dx  = [f\phi]_{-\infty}^{\infty} - \int_{-\infty}^{\infty} f\phi' \, dx = - \int_{-\infty}^{\infty} f\phi' \, dx = - f[\phi']$.
To approximate the Dirac delta function $\delta_c$, we use the Heaviside function $H_c$ at $c$ defined by $H_c(x) = 0$ if $x < c$ and $H_c(x) = 1$ if $x \geq c$.
Particularly, both $\delta_c$ and $H_c$ are distributions, and we can show that for any \(\phi \in \mathcal{S}\),
\(
(H_c)'[\phi] = - \int_{-\infty}^{\infty} H_c \phi' \, dx = - \int_{c}^{\infty} \phi' \, dx = \phi(c) = \int_{-\infty}^{\infty} \delta_c \phi = \delta_c[\phi]
\).
Thus, we interpret $\delta_c$ as the derivative of $H_c$ by this duality.

We consider the spectral approximation, particularly based on the Chebyshev polynomials. 
Let \(T_j(x)\) be the \(j\)th order Chebyshev polynomial of the first kind, defined on \([-1, 1]\), satisfying the three-term recurrence relation with \(T_0(x) = 1\) and \(T_1(x) = x\):
\(T_{j+1}(x) = 2xT_j(x) - T_{j-1}(x), \; j \geq 1.\)
The approximation begins by considering the set of Chebyshev polynomials, $\{T_j(x)\}_{j=0}^N$, which are used to span the polynomial space $B_N$. 
Let \(\{x_i\}_{i=0}^N\) be the Chebyshev Gauss-Lobatto collocation points on \([-1, 1]\), where \(x_i = -\cos\left(\frac{i\pi}{N}\right)\). 
The Chebyshev differentiation matrix \(D\) is defined such that its off-diagonal elements \(D_{ij}\) (for \(i \neq j\)) are given by \(\frac{c_i}{c_j} \frac{(-1)^{i+j}}{x_i - x_j}\).
The coefficients \(c_k\) are set to 2 if \(k = 0\) or \(k = N\), and 1 otherwise. 
The diagonal elements \(D_{ii}\) are expressed as \(-\frac{x_i}{2(1-x_i^2)}\) if \(i \neq 0 , N\).
For the cases \(i = 0 , N\), the values are specifically calculated with \(D_{00} = \frac{2N^2 + 1}{6}\) and \(D_{NN} = -\frac{2N^2 + 1}{6}\) \cite{hesthaven2007spectral}.  

The spectral approximation, $u_N(x) \in B_N$ of $u(x)$ is given by the projection of $u(x)$ to $B_N$, 
\(
u_N(x) \approx \sum_{j=0}^{N} \hat{u_j} T_j(x),
\)
where \( \hat{u}_j \) are the expansion coefficients \cite{hesthaven2007spectral}. Let $D\in \mathbb{R}^{(N+1)\times{(N+1)}}$ be the differentiation matrix based on the Gauss-Lobatto collocation points defined above. 
Note that $u_N(x)$ is the polynomial of degree at most $N$ given by the interpolation based on $\{u_N(x_i) \}_{i=0}^N$.
Then the derivative $u'_N(x_i)$ is exactly determined by
$\sum_{j=0}^{N} D_{ij} u(x_j)$.
Then with the derivative matrix, the Dirac delta function on the collocation points is given by the direct projection of the \(\delta\)-function onto \(B_N\) at the collocation points \(x_i\) is given by
\(\delta_N(x_i - c) = \sum_{j=0}^{N} D_{ij} H_c(x_j).\)
This representation is computationally efficient to implement.

\newpage
\section{Nonlinear differential equations perturbed with Dirac delta function
}
\label{sec:Nonlinear differential equations perturbed with Dirac delta function}
First we consider the time-independent problem on \(x \in [-1, 1]\) as given by:
\begin{equation}
\frac{d}{dx} \left( \frac{1}{2} u^2(x) \right) = \delta(x-c),
\label{eq:time_indep_1}
\end{equation}
with $u(-1) = \beta > 0$.
Let \( \mathbf{u} \)
denote the numerical solution vector on the collocation points, and $\mathbf{u}(x_i)$ be the $i^{th}$ component of $\mathbf{u}$.
Then, using the derivative matrix $D$, the Chebyshev method for equation \eqref{eq:time_indep_1} yields 
\begin{equation}
D \left( \frac{1}{2} \mathbf{u}^2 \right)  = D \mathbf{H}_c, 
\label{system}
\end{equation}
where
$\mathbf{u}^2 = \mathbf{u}*\mathbf{u} $, 
\(\mathbf{u}(x_0)= \beta > 0,\) and 
\(\mathbf{H}_c = ( H_c(x_0), \ldots, H_c(x_N) )^T\). 
Here, the superscript and $*$ operation mean the element by element product.
\begin{proposition}
    If \(\mathbf{u}\) is the solution of the system \eqref{system},
    then \(\mathbf{u}^2 = 2\mathbf{H}_c + \vec{\beta^2}\)
    where $\vec{\beta^2} = (\beta^2, \cdots, \beta^2)^T$.
\end{proposition}
\begin{proof}
Consider the Chebyshev differentiation matrix \( D \) of size \( (N+1) \times (N+1) \). 
We show that if \( D \mathbf{v} = 0\) for some vector \(\mathbf{v}=(v_0, v_1, \ldots, v_N)^T\), then \(\mathbf{v}\) is constant. 
If we define \( \tilde{D} \) as the submatrix of \( D \) obtained by excluding the first row and the first column, it is well-known that \( \tilde{D} \) is non-singular~\cite{funaro1988computing} so we know $\text{rank}(D) \geq N$.
Because the sum of the elements in each row of the Chebyshev matrix \( D \) is zero~\cite{trefethen2000spectral}, we have \(D \vec{\mathbf{1}} = 0\), with \(\vec{\mathbf{1}}=(1,1,\ldots,1)^T\).
Therefore, by rank-nullity theorem, we obtain \(\text{nullity}(D) = 1\) so the null space of \(D\) is \(\{\gamma \vec{1} \mid \gamma \in \mathbb{R} \}\).

Now, consider the the linear system \(D \left( \frac{1}{2} \mathbf{u}^2 \right)  = D \mathbf{H}_c \) with \(u(x_0)= \beta > 0\). From \(D \left( \frac{1}{2} \mathbf{u}^2 - \mathbf{H}_c \right) = 0\), there is a constant vector \(\mathbf{v}\) such that \(\frac{1}{2} \mathbf{u}^2 - \mathbf{H}_c = \mathbf{v}\). 
Evaluating at \(x_0\), \(\mathbf{v} = ( \frac{1}{2}\beta^2, \frac{1}{2}\beta^2, \ldots, \frac{1}{2}\beta^2)^T\) leads to \(\mathbf{u}^2 = 2\mathbf{H}_c + \vec{\beta^2}\), completing the proof.
\end{proof}
However, unlike Proposition~\ref{prop:numerical_method} that will be introduced later, the exact signs of the components of 
$\mathbf{u}$ cannot be determined.
Now we consider the following time-dependent differential equation
\begin{equation}
u_t + u_x = \delta(x-c), 
\quad
(x,t) \in [-1,1] \times [0, \infty), 
\label{eq:differential_equation}
\vspace{2mm}
\end{equation}
with the initial condition $u(x,0)=\beta$ and the boundary condition $u(-1,t) = \beta > 0$
for $(x,t) \in [-1,1] \times [0, \infty)$. 
In~\cite{jung2009note}, it was shown that the spectral approximation to the above equation can be obtained without oscillations, particularly for the steady-state solution.
However, for 
the nonlinear equation such as 
\begin{equation}
    u_t + \left(\frac{1}{2} u^2 \right)_x = \delta(x-c), 
    \quad
    (x,t) \in [-1,1] \times [0, \infty),
\label{eq:time_dep_2}
\end{equation}
with $u(x,0)=u(-1,t) = \beta > 0$ 
over the domain $(x,t) \in [-1,1] \times [0, \infty)$, whose  steady-state solution is given by $u = (\alpha - \beta) H_c + \beta$, where $\alpha = \sqrt{2 + \beta^2}$, the method in~\cite{jung2009note} yields still oscillatory approximation (see the left figure of  Figure~\ref{fig:oscillation_comparison}). 

\section{A consistent approach: the proposed method}\label{sec:A consistent approach: the proposed method}
In order to propose our method, consider Eqs.~\eqref{eq:differential_equation} and~\eqref{eq:time_dep_2}. 
The main difference between those two equations is that the first one is given in a linear form but the second in a nonlinear form for the advection term, i.e. $u u_x$ in the non-conservative form. 
The oscillations are induced by the nonlinear term of $f_x(u) = u u_x$. For Eq.~\eqref{eq:differential_equation}, the cancellation is obvious because the discretized version of Eq.~\eqref{eq:differential_equation} is 
$D {\mathbf{u}} = D {\bf H}_c$ for the steady-state solution. However, for Eq.~\eqref{eq:time_dep_2} which yields an oscillatory solution, the advection term $D\mathbf{u}$ is coupled with $\mathbf{u}$. Thus for the consistent method, we apply the same scaling to the right hand side as well, for which we exploit the definition of the Dirac delta function. 

The key idea of the proposed method is that for the solution \(u(x,t)\) we can use a scaled distribution for consistency as explained above \( \frac{u(x,t)}{u(c,t)} \delta_c\) if \(u(c,t) \neq 0\) (See Theorem~\ref{thm:proposed_method} for the positivity of $u(c,t)$).
That is, for any test function \(\phi\), we have 
\[
\left( \frac{u(x,t)}{u(c,t)} \delta_c \right) [\phi] = \int_{-\infty}^{\infty} \frac{u(x,t)}{u(c,t)} \phi \; \delta_c dx = \phi(c) = \delta_c [\phi].
\]
Hence, the distributions \( \frac{u(x,t)}{u(c,t)} \delta_c \) and \( \delta_c \) are equivalent when \( u(c,t) \neq 0 \) in the distribution sense.
Now, we consider the following differential equation to address Eq.~\eqref{eq:time_dep_2}:
\begin{equation}
u_t + \left(\frac{1}{2} u^2 \right)_x = \frac{u(x,t)}{u(c,t)} \delta(x-c), 
\quad
(x,t) \in [-1,1] \times [0, \infty),
\label{eq:adjusting_differential_equation}
\end{equation}
with the initial and boundary conditions 
$ u(x,0) = u(-1,t) = \beta > 0$.
Since the point \(c\) is not guaranteed to coincide with a collocation point, consider the two points \(x_{i^*}, x_{i^*+1}\) such that \(x_{i^*} < c \leq x_{i^*+1}\).
For this case, we use the average value of ${\mathbf{u}}(x_{i^*})$ and ${\mathbf{u}}(x_{i^*+1})$, which yields the following discretized equation for the steady-state solution of Eq.~\eqref{eq:adjusting_differential_equation}:
\begin{equation}
\mathbf{u} * D \mathbf{u} = \frac{2 \mathbf{u}}{\mathbf{u}(x_{i^*}) + \mathbf{u}(x_{i^*+1})} * D \mathbf{H_c}, 
\label{eq:proposed_numerical}
\end{equation}
where $\mathbf{u}(x_0) = \beta > 0$.
Note that we can use the average in denominator in Eq.~\eqref{eq:proposed_numerical} because $\mathbf{u}_n(x_{i^*})$ and $\mathbf{u}_n(x_{i^*+1}$) are positive as shown in Theorem~\ref{thm:proposed_method}. 

\vspace{4mm}
\begin{proposition}\label{prop:numerical_method}
For Eq.~\eqref{eq:proposed_numerical}, if \( \mathbf{u}(x_i) \neq 0 \) for all \( i \), then the solution is exactly obtained by 
\( {\mathbf{u}} = (\alpha - \beta ) \mathbf{H}_c + \vec{\beta} \),
where \(\alpha = \pm  \sqrt{2 + \beta^2}\).
\end{proposition}
\vspace{-6mm}
\begin{proof}
Assume that \( \mathbf{u}(x_i) \neq 0 \) for all \( i \).
Then, from \( \mathbf{u} * (D \mathbf{u} - \frac{2}{\mathbf{u}(x_{i^*}) + \mathbf{u}(x_{i^*+1})} D \mathbf{H}_c ) = 0 \), we have 
\(
D \mathbf{u} - \frac{2}{\mathbf{u}(x_{i^*}) + \mathbf{u}(x_{i^*+1})} D \mathbf{H}_c = 0 
\)
and so \( D \left(\mathbf{u} - \frac{2}{\mathbf{u}(x_{i^*}) + \mathbf{u}(x_{i^*+1})} \mathbf{H}_c \right) = 0 \). 
Now, we obtain \( \mathbf{u} - \frac{2}{\mathbf{u}(x_{i^*}) + \mathbf{u}(x_{i^*+1})} \mathbf{H}_c = \mathbf{v} \), where \(\mathbf{v}\) is a constant vector.
Since the first components of \( \mathbf{u} \) and \( \mathbf{H}_c \) are \(\beta\) and \(0\), respectively, \(\mathbf{v} = \vec{\beta}\) and we have 
\(
 \mathbf{u} = \frac{2}{\mathbf{u}(x_{i^*}) + \mathbf{u}(x_{i^*+1})} \mathbf{H}_c + \vec{\beta}.
\)
This implies that 
\( \mathbf{u}(x_0) = \mathbf{u}(x_1) = \ldots = \mathbf{u}(x_{i^*}) = \beta \).
Moreover, we have \( \mathbf{u}(x_{i^*+1}) = \frac{2}{\mathbf{u}(x_{i^*}) + \mathbf{u}(x_{i^*+1})} + \beta \), resulting in 
\( \mathbf{u}(x_{i^*+1}) = \pm\sqrt{2 + \beta^2} \).
Therefore, we have
\[
\mathbf{u} = \frac{2}{\mathbf{u}(x_{i^*}) + \mathbf{u}(x_{i^*+1})} \mathbf{H}_c + \vec{\beta} = (\alpha - \beta ) \mathbf{H}_c + \vec{\beta},
\]
where \(\alpha = \pm  \sqrt{2 + \beta^2}\).
\end{proof}

\vspace{4mm}
\begin{theorem}
\label{thm:proposed_method}
    Consider the following discretized problem with the Euler method for time integration: 
    \[
    \mathbf{u}_{n+1} = \mathbf{u}_{n} + dt_n \left( -
    \mathbf{u}_{n} * D \mathbf{u}_{n} + \frac{2 \mathbf{u}_{n}}{\mathbf{u}_{n}(x_{i^*}) + \mathbf{u}_{n}(x_{i^*+1})} * D \mathbf{H}_c \right), 
    \]
    where $\mathbf{u}_{0} = \vec{\beta}$ and $\mathbf{u}_{n}(x_0) = \beta > 0$, and \(dt_n\) is the time stepping \(\frac{C \min(dx_i)}{\lVert \mathbf{u}_n \rVert_{\infty}}\) for some constant \(C > 0\) where $dx_i = x_{i+1}-x_{i}$.
    If the sequence $\{\mathbf{u}_{n}\}$ converge to $\mathbf{u} = \lim_n \mathbf{u}_{n}$ such that $\mathbf{u}(x_i) \neq 0$ for all $i$, then  $\mathbf{u}$ is the steady-state solution of Eq.~\eqref{eq:time_dep_2}.
\end{theorem}
\begin{proof}
By taking the limit of the equation, we have
\(
    \mathbf{u} = \mathbf{u} + dt_n \left( -
    \mathbf{u} * D \mathbf{u} + \frac{2 \mathbf{u}}{\mathbf{u}(x_{i^*}) + \mathbf{u}(x_{i^*+1})} * D \mathbf{H}_c \right). 
\)
Therefore, by Proposition~\ref{prop:numerical_method}, we know that 
\(
\mathbf{u} = (\alpha - \beta ) \mathbf{H}_c + \vec{\beta},
\)
where \(\alpha = \pm  \sqrt{2 + \beta^2}\).

Now, we show that $\mathbf{u}_n(x_{i^*}),\mathbf{u}_n(x_{i^*+1})>0$ for all $n\ge 0$ implying that $\alpha = \sqrt{2+\beta^2}$.
Suppose that $\mathbf{u}_k(x_{i^*})$ and $\mathbf{u}_k(x_{i^*+1})$ are positive.
Then $\mathbf{u}_{k+1}(x_{i^*})$ satisfies the following
\begin{align*}
\mathbf{u}_{k+1}(x_{i^*}) &= \mathbf{u}_k(x_{i^*})\left(1-dt_k(Du)(x_{i^*})+\frac{2dt
_k}{\mathbf{u}_k(x_{i^*})+\mathbf{u}_k(x_{i^*+1})}(D\mathbf{H}_c)(x_{i^*}) \right).
\end{align*} 
The value $\frac{2dt_k}{\mathbf{u}_k(x_{i^*})+\mathbf{u}_k(x_{i^*+1})}(D\mathbf{H}_c)(x_{i^*}))$ is positive, since the heaviside function $H_c$ changes the function value in $(x_{i^*},x_{i^*+1}]$, $D\mathbf{H}_c(x_{i^*})$ is positive.
Take $C=\min\{\frac{1}{2\|D\|_{\infty}}, \frac{1}{2}\}.$
Then it follows that
\(
dt_k\|D\mathbf{u}_k\|_{\infty}
    \le \frac{\min(dx_i)}{2\|D\|_{\infty}\|\mathbf{u}_k\|_{\infty}}\|D\|_{\infty}\|\mathbf{u}_k\|_{\infty}\le \frac{1}{2}.
\)
Therefore we can get from the above inequality that
\(
1-dt_k(D\mathbf{u}_k)(x_{i^*})+\frac{2dt_n}{\mathbf{u}_k(x_{i^*})+\mathbf{u}_k(x_{i^*+1})}(D\mathbf{H}_c)(x_{i^*})
> 1-\frac{1}{2}>0,
\)
which implies the positivity of $\mathbf{u}_{k+1}(x_{i^*}).$
We can also easily show that $\mathbf{u}_{k+1}(x_{i^*+1})$ is also positive because $D\mathbf{H}_c(x_{i^*+1})$ is positive.
Since $\mathbf{u}_0(x_{i^*})=\mathbf{u}_0(x_{i^*+1})=\beta>0$, by induction,
it holds that $\mathbf{u}_n(x_{i^*})$ and $\mathbf{u}_n(x_{i^*+1})$ are positive for all $n\ge0$.
\end{proof}

We extend our analysis to a more general Burgers' type nonlinear differential equation with a nonlinear convection term of higher degree given by:
\begin{equation}
u_t + \left(\frac{1}{m} u^m \right)_x = \delta(x-c, 
\quad
(x,t) \in [-1,1] \times [0, \infty), 
\label{eq:general_differential_equation}
\end{equation}
with $ u(x,0) = u(-1,t) = \beta > 0$. 
This equation, a generalization of the classic Burgers' equation, introduces a nonlinear dependency in the convection term \( u^m \), modeling more complex phenomena such as enhanced fluid flows or wave propagations, which may involve shock wave formations and steep gradient developments.
For Eq.~\eqref{eq:general_differential_equation}, we consider the following discretized problem with the Euler method for time integration:
\begin{equation}
\mathbf{u}_{n+1} = \mathbf{u}_{n} + dt_n \left( -
\mathbf{u}_{n}^{m-1} * D \mathbf{u}_{n} + \frac{m \mathbf{u}_{n}^{m-1}}{\sum_{j=1}^m \mathbf{u}_{n}^{m-j}(x_{i^*})\mathbf{u}_{n}^{j-1}(x_{i^*+1})}  * D \mathbf{H}_c \right), 
\label{eq:proposed_general_numerical}
\end{equation}
where $\mathbf{u}_{0} = \vec{\beta}$ and $\mathbf{u}_{n}(x_0) = \beta > 0$, and \(dt_n\) is the time stepping \(dt_n=\frac{C \min(dx_i)}{\lVert \mathbf{u_n^{m-1}} \rVert_{\infty}}\) for some constant \(C > 0\), where \(dx_i=x_{i+1}-x_{i}\).
We do not use the simple average but the average with multiplication because $\alpha$ and $\beta$ satisfy that $m=\alpha^m-\beta^m=(\alpha-\beta)(\alpha^{m-1}+\alpha^{m-2}\beta+\cdots+\alpha\beta^{m-2}+\beta^{m-1})$ and the steady state solution contains $(\alpha-\beta)\mathbf{H}_c$ terms.

\section{Numerical results}\label{sec:Numerical results}
In this section, we demonstrate through various examples that the proposed method accurately approximates the numerical solution to equations with sngular source terms.

\subsection{One-dimensional Burgers' equation with a point source term}\label{subsec:One-dimensional Burgers' equation}

\begin{figure}[htbp]
    \centering
    \hfill
    \begin{minipage}{0.45\textwidth}
        \centering
        \includegraphics[width=\linewidth]{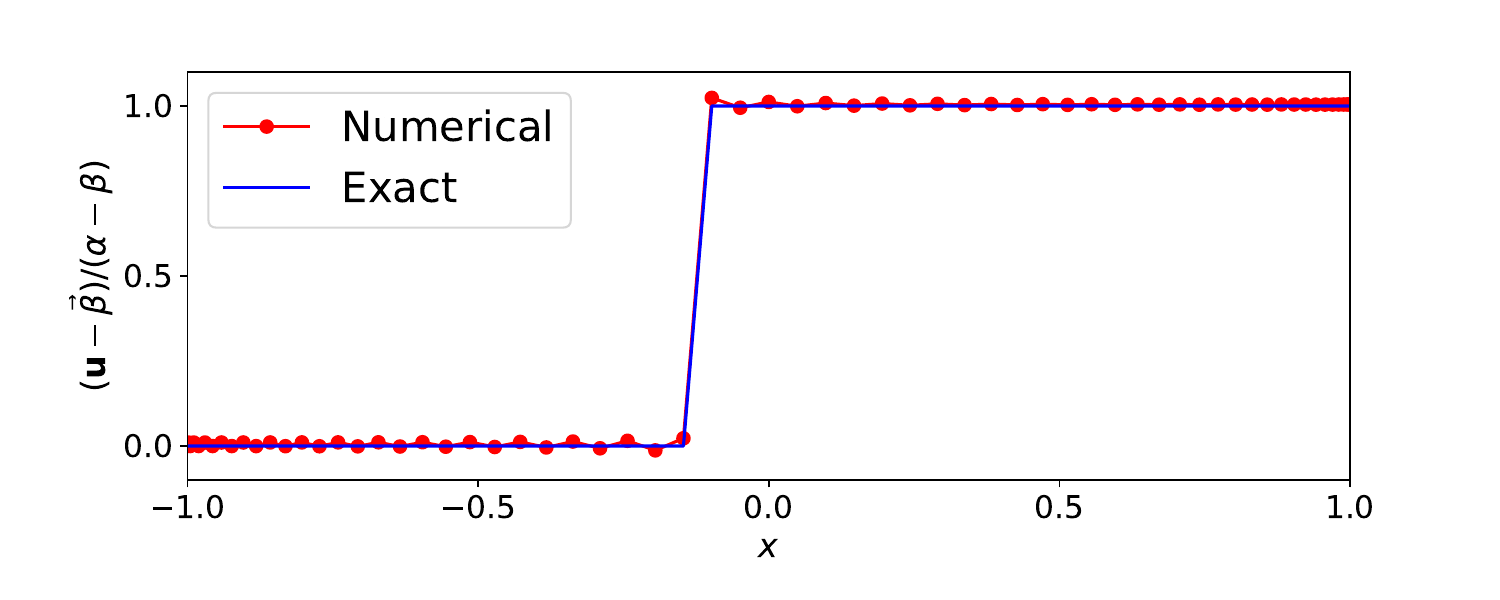} 
    \end{minipage}\hfill
    \begin{minipage}{0.45\textwidth}
        \centering
        \includegraphics[width=\linewidth]{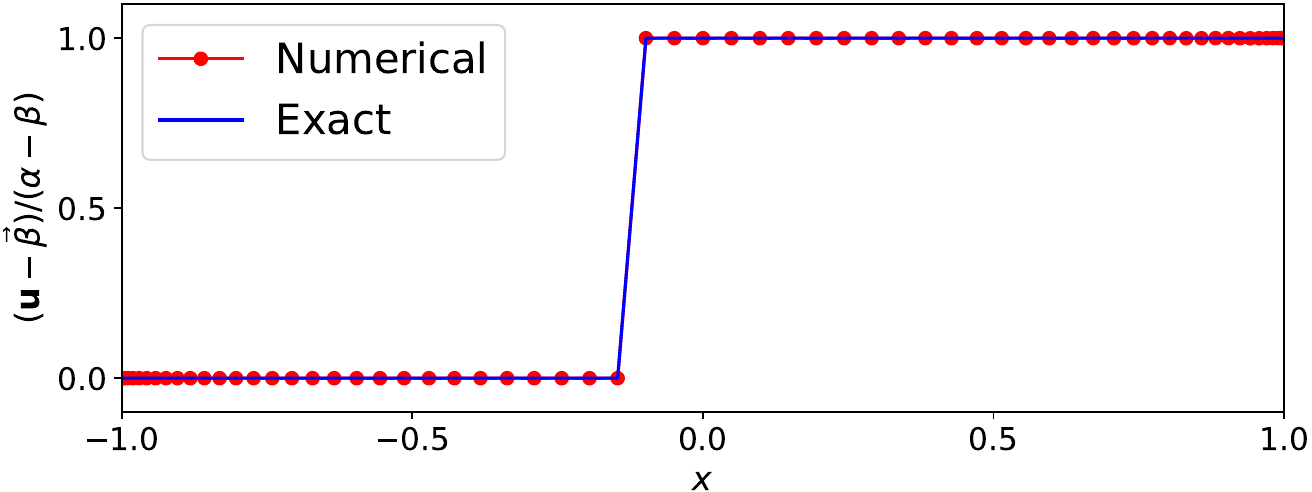}
    \end{minipage}\hfill
    \hfill
    \caption{\((\mathbf{u} - \vec{\beta}) / (\alpha - \beta)\) versus $x$, where \(\mathbf{u}\) represents the numerical solution of Eq.~\eqref{eq:time_dep_2} with \(c = -0.1\), \(\beta = 2\),\ \(C=0.5\) and \(N=64\). Left: The previous Schwartz method at time $t= 1.63$. Right: The proposed method at time $t= 1.63$.
    The exact \(\mathbf{H}_c\) is given in blue solid line and the approximation is shown with red dots.
    }
    \label{fig:oscillation_comparison}
\end{figure}

Figure~\ref{fig:oscillation_comparison} shows the numerical approximation, \((\mathbf{u}_{n} - \vec{\beta}) / (\vec{\alpha} - \vec{\beta})\), to Eq.~\eqref{eq:time_dep_2} at $t = 1.63$. The left and right figures show the results of the previous Schwartz method and the proposed method, respectively. 
The exact \(\mathbf{H}_c\) is given in blue solid line and the numerical approximation is in red on the grids.  
As shown in Figure \ref{fig:oscillation_comparison}, it is evident that the proposed method significantly mitigates oscillations for nonlinear problems on the grids compared to the previous method in the left figure.
In this note, we used the Euler method for time integration.
We obtain the similar result with the high order time integration methods such as the Runge-Kutta method.

\subsection{Higher-degree nonlinear equations}\label{subsec:Higher-degree nonlinear equations}

\begin{figure}[htbp]
    \centering
    \begin{minipage}{0.5\textwidth}
        \centering        \includegraphics[width=\textwidth]{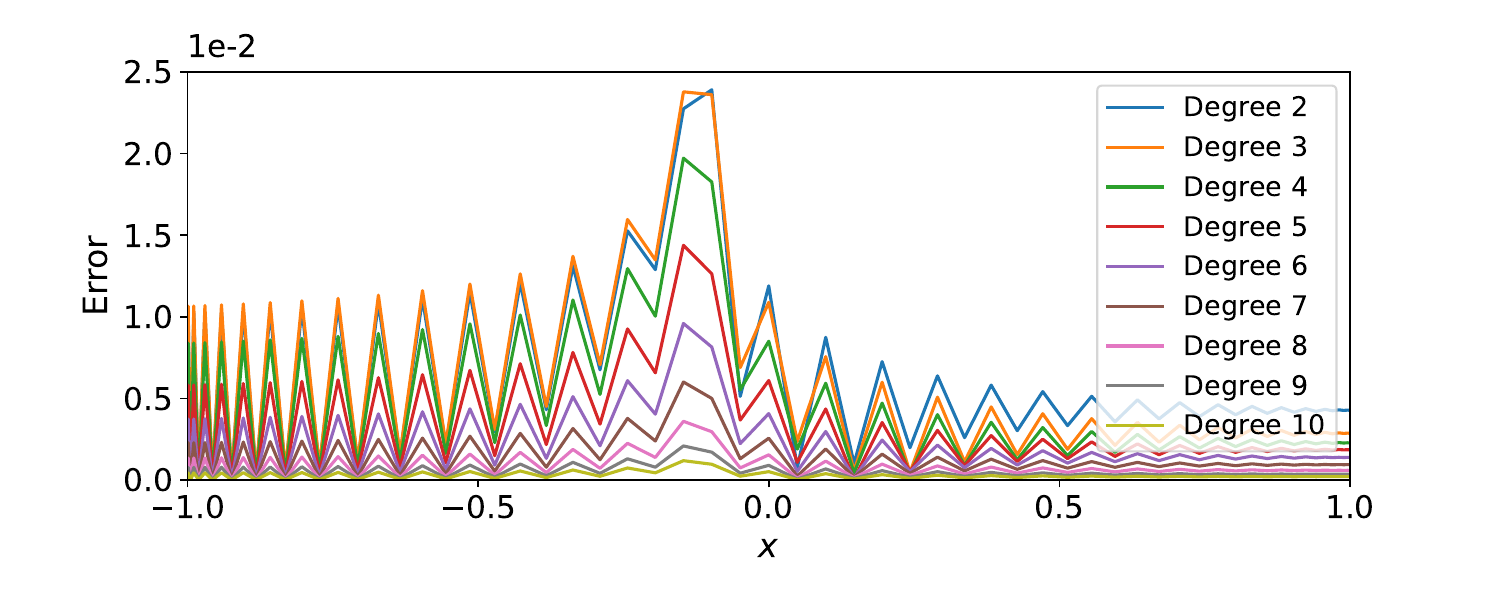} 
    \end{minipage}\hfill
    \begin{minipage}{0.5\textwidth}
        \centering        \includegraphics[width=\textwidth]{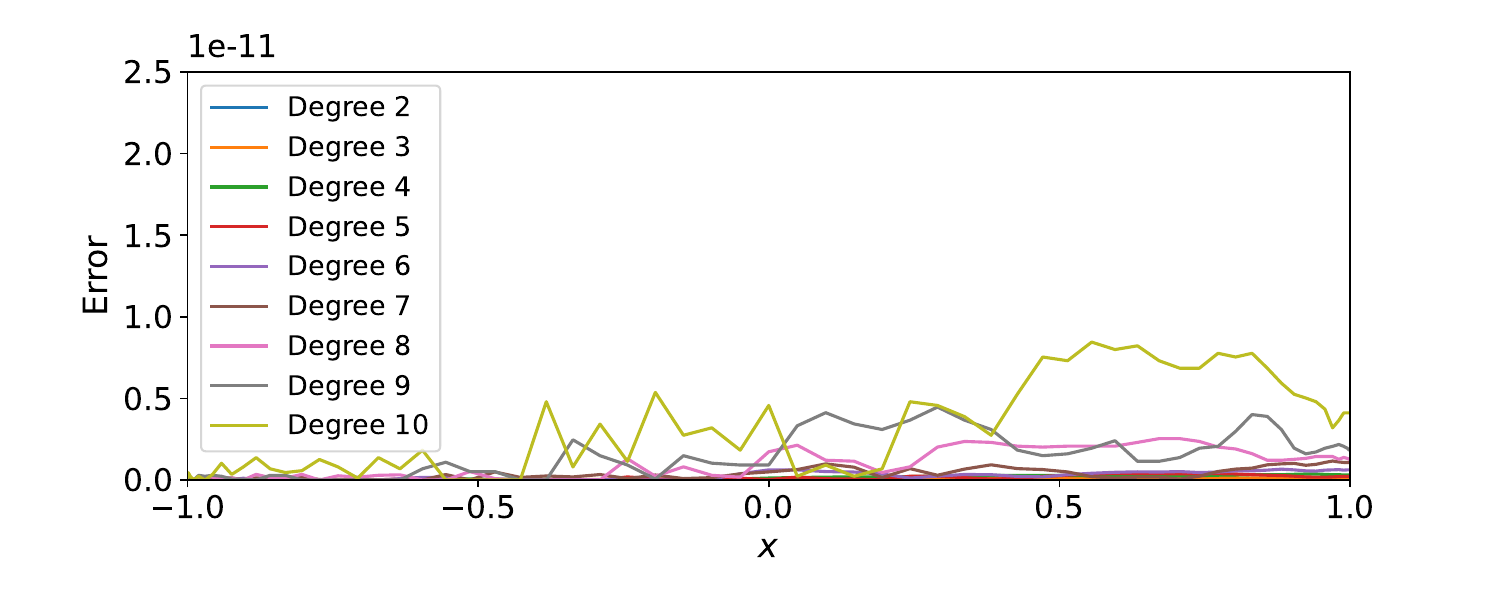}
    \end{minipage}
    \caption{\(|\mathbf{H}_c - (\mathbf{u}-\vec{\beta})/(\alpha - \beta)|\) for various degrees \(m = 2, 3, \ldots, 10\), for the previous method (Left) and the proposed method (Right) with \(c = -0.1\), \ \(\beta = 2\), \ \(C=0.5\) and \(N=64\) of the numerical method~(\ref{eq:proposed_general_numerical}).}
    \label{fig:method_degree_comparison}
\end{figure}

Let \( m \) in Eq.~\eqref{eq:general_differential_equation} be referred to as the degree of the equation.
The absolute value of the difference between the Heaviside function \( \mathbf{H}_c \) and \( (\mathbf{u}-\vec{\beta})/(\alpha - \beta) \) is shown in Figure~\ref{fig:method_degree_comparison}.
The left of Figure~\ref{fig:method_degree_comparison} represents the results with various degrees by the previous method, while the right shows the results by the proposed method.
The ratios of the maximum errors between the previous and proposed methods for each degree are provided in Table~\ref{tab:degree_ratio}.
\textcolor{black}{ Table \ref{tab:degree_ratio} clearly shows that the proposed method performs better.}
For example, in the degree \(2\) case, the max error ratio between the previous and proposed methods, denoted by \(\max\{|\mathbf{H}_c - (\mathbf{u}^{\text{prev}}-\vec{\beta})/(\alpha-\beta)| \} / \max\{|\mathbf{H}_c - (\mathbf{u}^{\text{new}}-\vec{\beta})/(\alpha-\beta)| \}\), is $287.91 \times 10^{9}$, where \(\mathbf{u}^{\text{prev}}\) and \(\mathbf{u}^{\text{new}}\) are the numerical solutions of the previous and proposed methods, respectively.

\begin{table}[htbp]
    \centering
    \begin{tabular}{|c|c|c|c|c|c|c|c|c|c|}
    \hline
    Degree & 2 & 3 & 4 & 5 & 6 & 7 & 8 & 9 & 10 \\ \hline
    Max Error Ratio & 287.91 & 150.02 & 53.63 & 36.00 & 14.44 & 5.16 & 1.41 & 0.46 & 0.14 \\ \hline
    \end{tabular}
    \vspace{2mm}
    \caption{
    Comparison of the degree-to-max-error ratio between previous and proposed methods, expressed in units of
    \(10^{9}\).
    }
    \label{tab:degree_ratio}
\end{table}

\subsection{Differential equation with multiple source terms} \label{subsec:multiple delta}
Singular PDEs with multiple point source terms arise in various fields, including the cell elasticity \cite{refId0}, water pollution \cite{hamdi2007identification} and acoustic waves \cite{Chen_2020}. However, solving such PDEs remains a significant challenge. In this section, we focus on a nonlinear one-dimensional PDE with multiple delta function source terms, formulated as follows:
\begin{equation}
\label{eq:burgers multiple source}
    u_t + u u_x = \sum_{i=1}^k a_i\delta(x-c_i),   
    \quad
    (x,t) \in [-1,1] \times [0, \infty),
\end{equation}
with $u(x,0) = u(-1,t) = \beta > 0$ for real numbers $a_i > 0$ and $c_1 < c_2 < \ldots < c_k$. 
Consider numbers $b_i = \sum_{m=1}^i 2a_m$ and $b_0 =0$ for all $i$.
Then, the steady-state solution of Eq.~\eqref{eq:burgers multiple source} is $\beta + \sum_{i=1}^k (\sqrt{b_i + \beta^2} - \sqrt{b_{i-1} + \beta^2 })H_{c_i}$, where $H_{c_i}$ is the Heaviside function at $x=c_i$.

To obtain a numerical solution of Eq.~\eqref{eq:burgers multiple source}, we apply the Gaussian approximation (left) and the proposed method \eqref{eq:proposed_numerical} (right) to each delta function $\delta(x - c_i)$ and the Euler method with the same time step as Section~\ref{subsec:Higher-degree nonlinear equations} similar to Theorem \ref{thm:proposed_method}. 
\textcolor{black}{For the Gaussian approximation, the Gaussian function is used with the mean $\mu=c_i$ and the standard deviation $\sigma=0.02$ as the approximation of a delta function $\delta(x-c_i)$.
Note that the Gaussian function ($\mu, \sigma$) should be discretized in a way that ensures the quadrature sum converges to unity.
}
As shown in Figure \ref{fig:mutiple_delta}, we can see that the proposed method \eqref{eq:proposed_numerical} successfully eliminates oscillations in the numerical solution of the one-dimensional Burgers' equation with multiple point sources and provide a non-dissipative numerical solution (right) while the Gaussian approximation yields a dissipative solution (left).
\begin{figure}[htbp]
    \centering
    \hfill
    \begin{minipage}[b]{0.5\textwidth}
        \includegraphics[width=\textwidth]{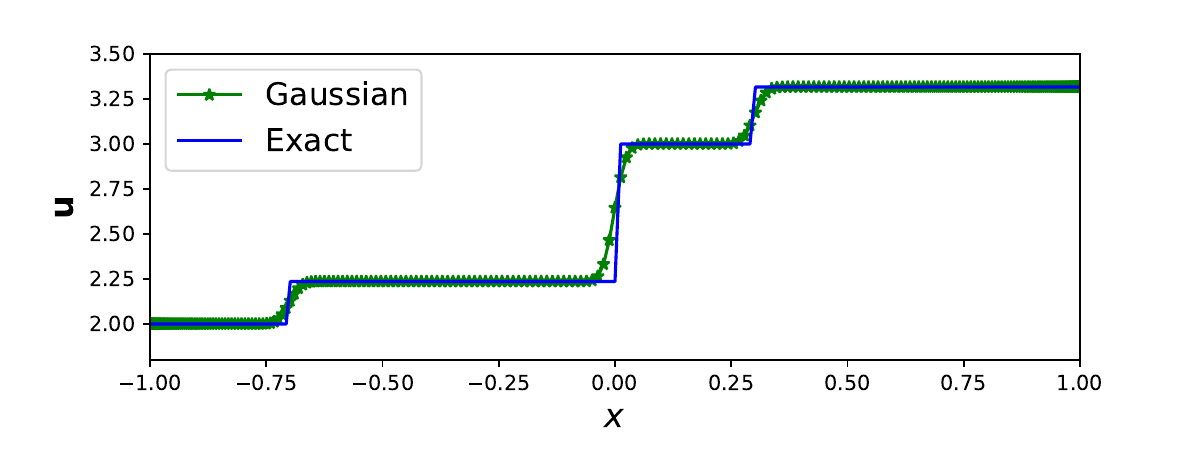}
    \end{minipage}\hfill
    \begin{minipage}[b]{0.5\textwidth}
        \includegraphics[width=\textwidth]{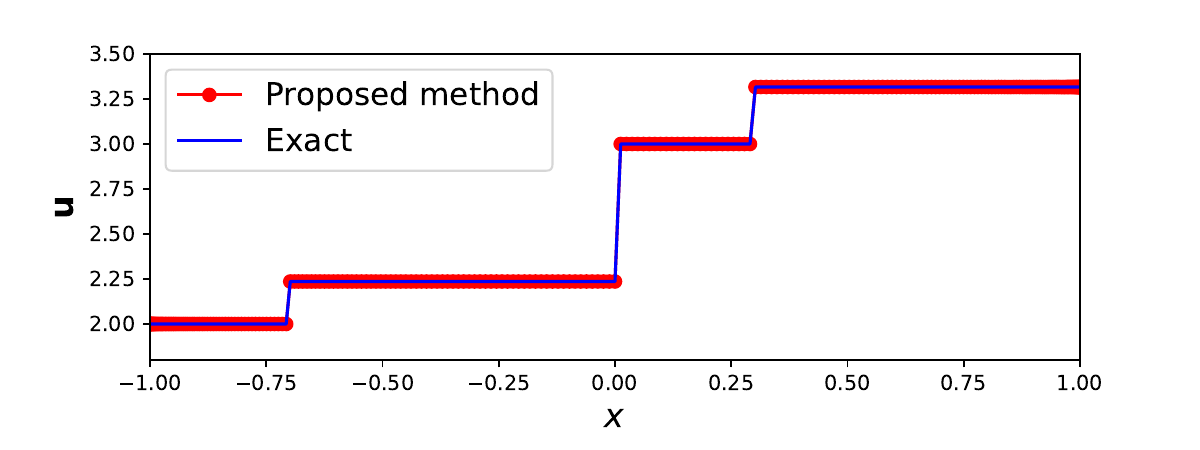}
    \end{minipage}\hfill
    \hfill
    \caption{Numerical simulation $\mathbf{u}$ of Eq.~\eqref{eq:burgers multiple source} with $N=256$, $\beta=2$, $C=0.5$ and the source term $0.5\delta(x+0.7) + 2\delta(x) + \delta(x-0.3)$ by Gaussian approximation with the standard derivation $\sigma=0.02$ (left, green) and the proposed method (right, red) at time $t=1.5$. The exact solution $2 + (\sqrt{5}-2)H_{-0.7}+ (3-\sqrt{5})H_{0} + (\sqrt{11}-3)H_{0.3}$ is given in blue solid line.}
    \label{fig:mutiple_delta}
\end{figure}

\subsection{Two-dimensional differential equations}\label{subsec:Two-dimensional equations}
\textcolor{black}{The proposed method can be extended to two-dimensional differential equations.}
First, we solve the two-dimensional advection equation, given by
\begin{equation}
u_t + \textcolor{black}{a}u_x + \textcolor{black}{b}u_y
= \delta_{(0,0)}(x,y), 
\quad
\textcolor{black}{(x,y,t) \in [-1,1]^2 \times [0, \infty)},
\label{eq:2d_advec}
\end{equation}
with positive $a$ and $b$, consequently the boundary conditions $u(-1,y,t) = \sin(\pi(-1-t))\sin(\pi(y-t))$, 
$u(x,-1,t) = \sin(\pi(x-t))\sin(\pi(-1-t))$, and 
the initial condition $u(x,y,0) = \sin(\pi x) \sin(\pi y)$.
\textcolor{black}{By considering the characteristic curve  $X_{(a,b)}(t) = (x(t), y(t)) = (at+c, bt+d)$ for some $c,d \in \mathbb{R}$,
we have 
\(
\frac{d}{dt} u(X_{(a,b)}(t), t) = \frac{d}{dt} u(x(t), y(t), t) = u_t + au_x + bu_y = \delta_{(0,0)}(X_{(a,b)}(t)).
\)
Therefore, the exact solution can be determined to be $u = \sin(\pi(x-t)) \sin(\pi(y-t)) + 1$ if $bx=ay \geq 0$, and $u = \sin(\pi(x-t)) \sin(\pi(y-t))$ otherwise.
To find the two-dimensional numerical solution $\mathbf{u}$ of Eq.~\eqref{eq:2d_advec}, consider the test function $T(x,y) = 1$ if $bx=ay \geq 0$ and $T(x,y)=0$ otherwise.
Let $\mathbf{T}$ be the matrix representation of the function $T$ over the grid. 
Then, we approximate the delta function $\delta_{(0,0)}$ using the Chebyshev matrix $D$ as $a\mathbf{T}D^t + bD \mathbf{T}$. 
This is meant to approximate $(a\frac{\partial}{\partial x} + b\frac{\partial}{\partial y}) T$.
This is consistent with the use of the Heaviside function in the one-dimensional case.
Figure~\ref{fig:sine_schwartz} shows the numerical solution to Eq.~\eqref{eq:2d_advec} with the delta function $\delta_{(0,0)}$ as $\mathbf{T}D^t + D\mathbf{T}$ for $a=b=1$.}

\begin{figure}[htbp]
    \centering
    \begin{subfigure}[b]{0.33\textwidth}
        \includegraphics[width=\textwidth]{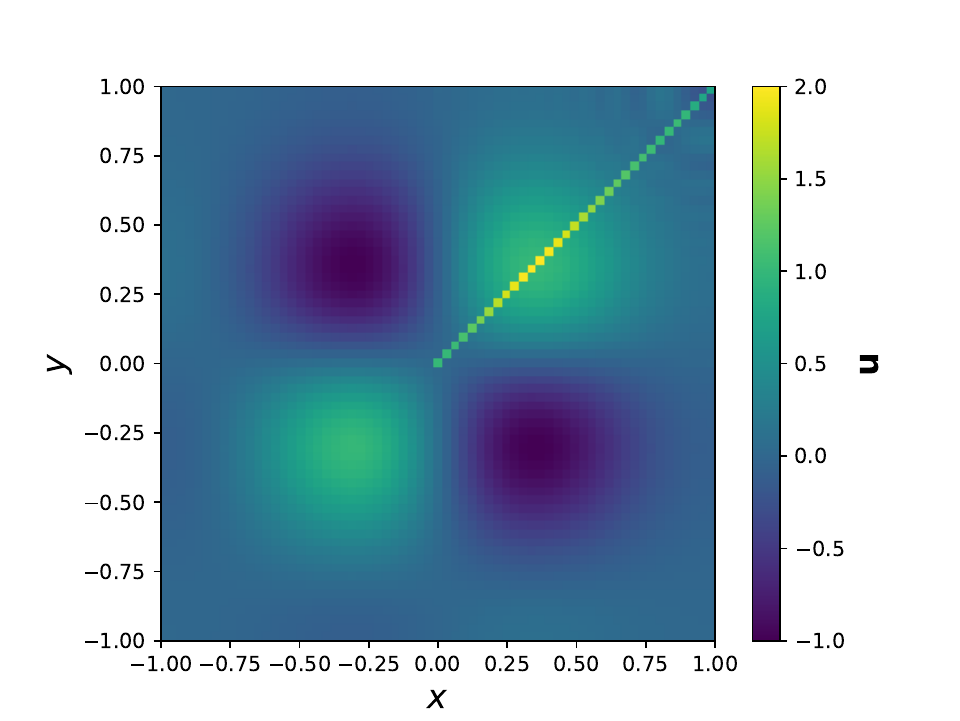}
        \caption{Numerical solution of Eq.~\eqref{eq:2d_advec}}
        \label{fig:sine_schwartz}
    \end{subfigure}\hfill
    \begin{subfigure}[b]{0.33\textwidth}
        \includegraphics[width=\textwidth]{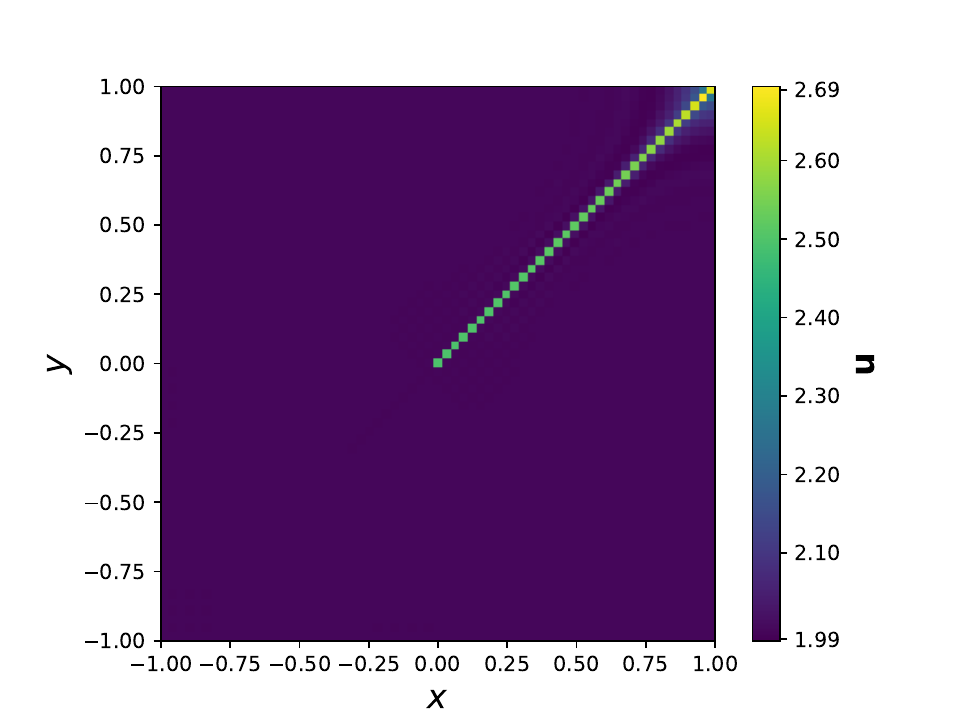}
        \caption{Numerical solution of Eq.~\eqref{eq:2d-nonlinear}}
        \label{fig:schwartz_2d_bad}
    \end{subfigure}\hfill
    \begin{subfigure}[b]{0.33\textwidth}
        \includegraphics[width=\textwidth]{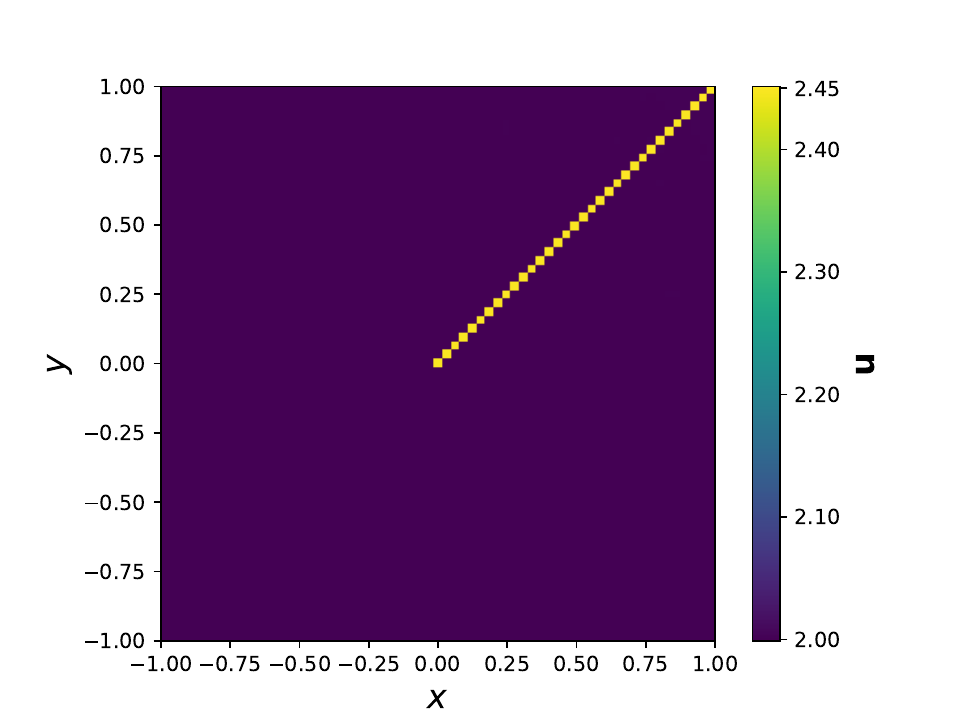}
        \caption{Numerical solution of Eq.~\eqref{eq:2d-nonlinear}}
        \label{fig:schwartz_2d_good}
    \end{subfigure}
    \caption{Numerical solutions of Eqs.~\eqref{eq:2d_advec} with $a=b=1$ and~\eqref{eq:2d-nonlinear}  by approximating the delta function $\delta_{(0,0)}$ as $\mathbf{T}D^t+D\mathbf{T}$ in (a) and (b), and as $\frac{2\mathbf{u}}{\mathbf{u}(x_i^*,x_i^*)+\mathbf{u}(x_{i+1}^*,x_{i+1}^*)}(\mathbf{T}D^t+D\mathbf{T})$ in (c) with \(\beta = 2\), \ \(C=0.5\) and \(N=64\) at time $t= 1$.}
    \label{fig:2d_nonlinear_bad_good}
\end{figure}

The above approach can be applied to nonlinear problems as well. Consider the two-dimensional non-linear differential equation, given by
\begin{equation}
u_t + u u_x + u u_y = \delta_{(0,0)}(x,y),
\quad
(x,y,t) \in [-1,1]^2 \times [0, \infty),
\label{eq:2d-nonlinear}
\end{equation}
with $u(0,y,t) = u(x,0,t) = u(x,y,0) = \beta > 0$.
By considering the characteristic curve $X_{u}(t)$ with $\frac{d}{dt} X_u(t)=(u,u)$, the exact steady-state solution of Eq.~\eqref{eq:2d-nonlinear} goes to $u(x,y) = \sqrt{2+\beta^2}$ if $x=y\ge0$ and $\beta$ if not. 
To calculate the numerical solution, we use the proposed method \eqref{eq:proposed_numerical} to approximate the delta function $\delta_{(0,0)}$ as
$\frac{2\mathbf{u}}{\mathbf{u}(x_i^*,x_i^*)+\mathbf{u}(x_{i+1}^*,x_{i+1}^*)}(\mathbf{T}D^t+D\mathbf{T})$
and use the Euler method with the same time stepping as in Section~\ref{subsec:Higher-degree nonlinear equations}.
Figures~\ref{fig:schwartz_2d_bad} and~\ref{fig:schwartz_2d_good}, respectively, show the numerical solutions of Eq.~\eqref{eq:2d-nonlinear} by approximating the delta function $\delta_{(0,0)}$ as $\mathbf{T}D^t + D\mathbf{T}$ and as $\frac{2\mathbf{u}}{\mathbf{u}(x_i^*, x_i^*) + \mathbf{u}(x_{i+1}^*, x_{i+1}^*)} (\mathbf{T}D^t + D\mathbf{T})$.
Figure~\ref{fig:schwartz_2d_good} demonstrates that the solution converges exactly to $\mathbf{u}(x,y) = \sqrt{6}$ if $x=y \geq 0$ and to $\beta$ otherwise.

\subsection{Generalization}\label{subsec:Generalization to finite difference method with uniform points}

The proposed method can be further generalized to other collocation methods, such as the finite difference method. The key idea is that once the numerical derivative operator is defined properly, the proposed method can be applied in the same way as the Chebyshev method.
First, we experimentally demonstrate that the numerical solutions of Eqs.~\eqref{eq:2d_advec} and~\eqref{eq:2d-nonlinear} do not perform well when the source term is approximated using a two-dimensional  Gaussian function even though the quadrature sum of such a function over a two-dimensional grid converges to unity. In particular, the approximation becomes even less effective at Chebyshev Gauss–Lobatto collocation points. As an example, 
we use the first order forward difference matrix \(D\) on a uniform grid \(\{x_i\}_{i=0}^N\) over the interval \([-1,1]\). 
\textcolor{black}{
Figures~\ref{fig:sine_gaussian} and \ref{fig:gaussian_nonlinear} show the numerical approximations of Eqs.~\eqref{eq:2d_advec} and \eqref{eq:2d-nonlinear}, respectively, both with the delta function $\delta_{(0,0)}(x,y)$ regularized as the two-dimensional Gaussian function.
Figure~\ref{fig:Spectral_nonlinear_finitedifference} shows the approximation for Eq.~\eqref{eq:2d-nonlinear} with 
the delta function $\delta_{(0,0)}(x,y)$ implemented with the proposed method as $\frac{2\mathbf{u}}{\mathbf{u}(x_i^*,x_i^*)+\mathbf{u}(x_{i+1}^*,x_{i+1}^*)}(\mathbf{T}D^t+D\mathbf{T})$.
As shown in the figure, the Gaussian method does not yield accurate results. 
Figure~\ref{fig:Spectral_nonlinear_finitedifference} demonstrates that our proposed method can be applied to other collocation methods and yields the exact solution using a similar approach as employed with the spectral method above.
}

\begin{figure}[htbp]
    \centering
    \hfill
    \begin{subfigure}[b]{0.33\textwidth}
        \includegraphics[width=\textwidth]{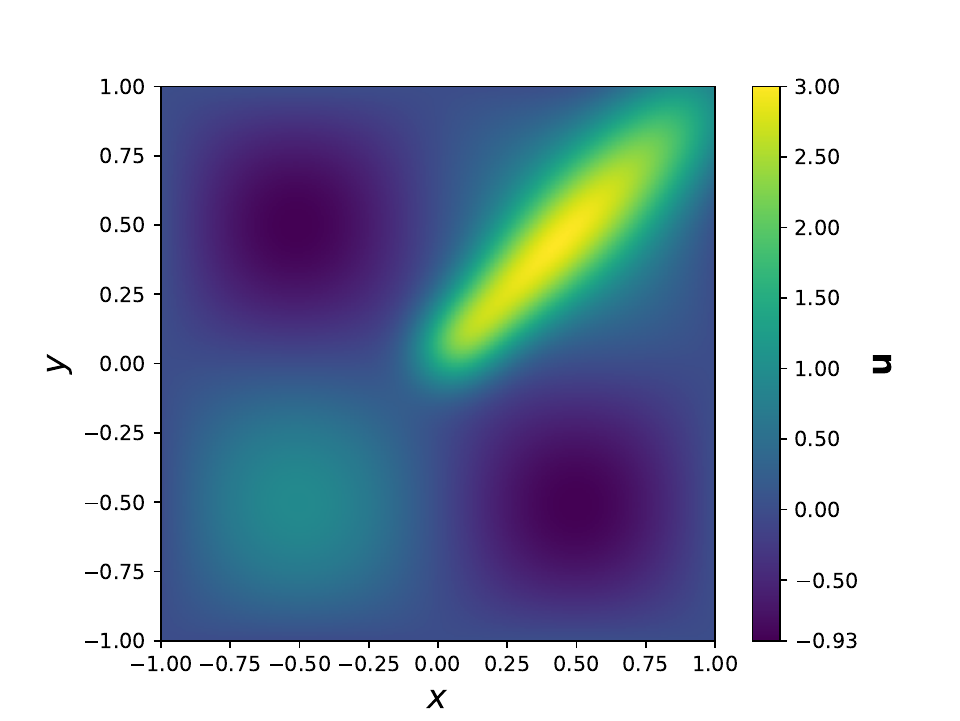}
        \caption{Numerical solution of Eq.~\eqref{eq:2d_advec}}
        \label{fig:sine_gaussian}
    \end{subfigure}\hfill
    \begin{subfigure}[b]{0.325\textwidth}
        \includegraphics[width=\textwidth]{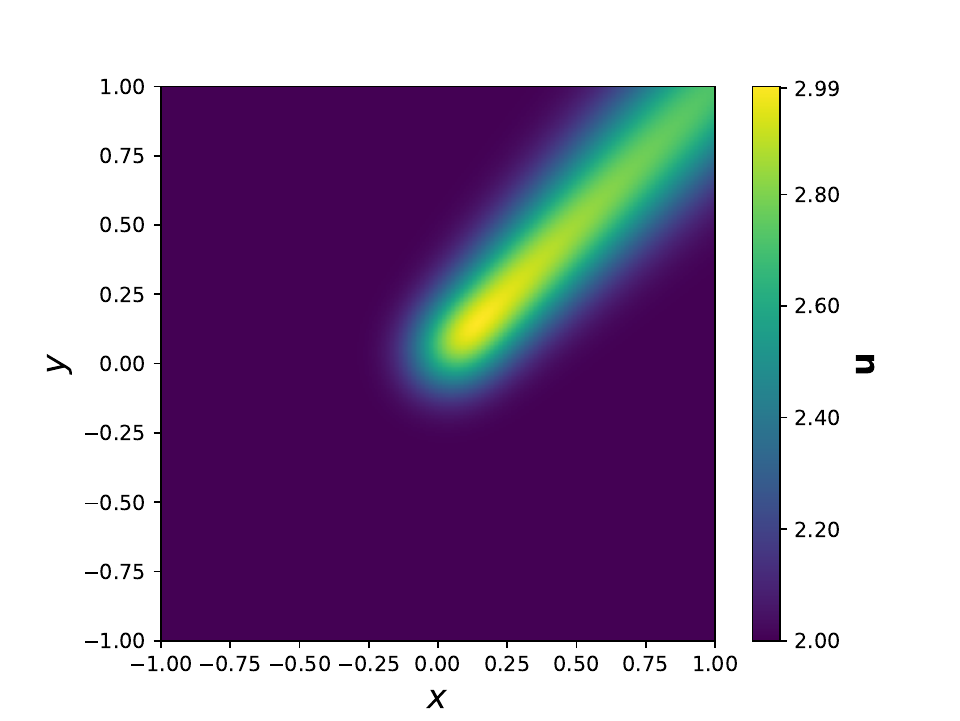}
        \caption{Numerical solution of Eq.~\eqref{eq:2d-nonlinear}}
        \label{fig:gaussian_nonlinear}
    \end{subfigure}\hfill
    \begin{subfigure}[b]{0.325\textwidth}
        \includegraphics[width=\textwidth]{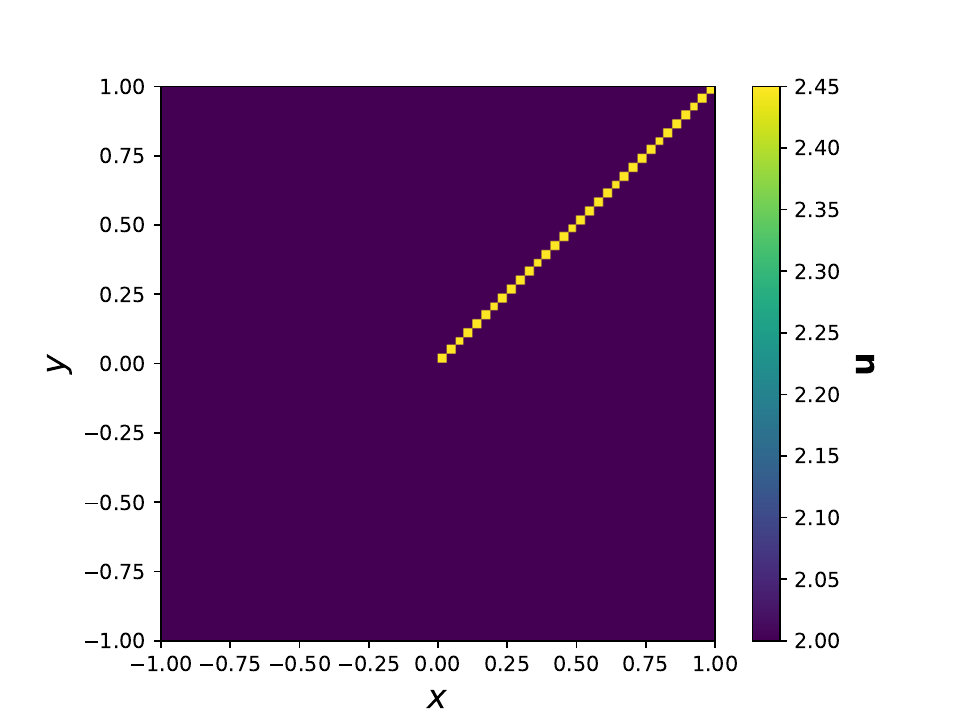}
        \caption{Numerical solution of Eq.~\eqref{eq:2d-nonlinear}}
        \label{fig:Spectral_nonlinear_finitedifference}
    \end{subfigure}\hfill
    \hfill
    \caption{Numerical solution of Eqs.~\eqref{eq:2d_advec} \textcolor{black}{with $a=b=1$} and~\eqref{eq:2d-nonlinear} by approximating the delta function $\delta_{(0,0)}$ as two-dimensional Gaussian functions at time $t=2$ with \(\beta = 2\), $N=128$ and the standard deviation $\sigma=0.1$ for figures (a) and (b). For Figure (c), the delta function $\delta_{(0,0)}$ is approximated $\frac{2\mathbf{u}}{\mathbf{u}(x_i^*,x_i^*)+\mathbf{u}(x_{i+1}^*,x_{i+1}^*)}(\mathbf{T}D^t+D\mathbf{T})$ at time $t=1$ with \(\beta = 2\), \ \(C=0.5\) and \(N=64\). Here $D$ is the forward difference matrix.}
    \label{fig:compare_gaussian}
\end{figure}

\section{Concluding remark}\label{sec:Concluding remark}
Singularly perturbed differential equations with Dirac-delta type sources present a significant computational challenge when seeking numerical solutions. Any high-order method suffers from Gibbs oscillations due to the discontinuous nature of the solutions. In this note, we demonstrated that a simple remedy based on Schwarz duality can eliminate the Gibbs phenomenon and yield exact solutions for collocation methods, even in nonlinear problems. First, we showed that the method is effective for high-order methods such as spectral methods, and then generalized it to difference equations. 
Our numerical examples for both linear and nonlinear problems, in both one- and two-dimensional cases, confirm that the proposed method is accurate and efficient. In future work, we will develop a systematic approach for constructing a consistent derivative matrix for higher-order approximations for general collocation methods and explore cases where the singular sources are dynamic.

\section*{Acknowledgments}
This work is supported by National Research Foundation (NRF) of Korea under the grant number 2021R1A2C3009648, POSTECH Basic Science Research Institute under the NRF grant number 2021R1A6A1A10042944 and partially NRF grant funded by the Korea government (MSIT) (No. RS-2023-00219980).

\newpage
\bibliography{reference} 
\bibliographystyle{elsarticle-num}

\end{document}